\documentclass[11pt]{article}

\title{Braid stability for periodic orbits of area-preserving surface diffeomorphisms}
\author{Michael Hutchings}
\date{}

\usepackage{amssymb}
\usepackage{latexsym}
\usepackage{amsmath}
\usepackage{amsthm}
\usepackage{amscd}
\usepackage{color}

\newcommand{\mc}[1]{{\mathcal #1}}

\numberwithin{equation}{section}

\newtheorem{theorem}{Theorem}[section]

\newtheorem{proposition}[theorem]{Proposition}
\newtheorem{corollary}[theorem]{Corollary}
\newtheorem{lemma}[theorem]{Lemma}
\newtheorem{lemma-definition}[theorem]{Lemma-Definition}

\theoremstyle{definition}
\newtheorem{definition}[theorem]{Definition}

\newtheorem{remark}[theorem]{Remark}

\newtheorem{example}[theorem]{Example}

\newcommand{\floor}[1]{\left\lfloor #1 \right\rfloor}

\newcommand{\R}{{\mathbb R}}

\newcommand{\Z}{{\mathbb Z}}

\newcommand{\op}{\operatorname}

\newcommand{\M}{\mc{M}}

\newcommand{\union}{\bigcup}

\newcommand{\bpm}{\begin{pmatrix}}
\newcommand{\epm}{\end{pmatrix}}

\renewcommand{\epsilon}{\varepsilon}

\begin{document}

\maketitle

\begin{abstract}
We consider an area-preserving diffeomorphism of a compact surface, which is assumed to be an irrational rotation near each boundary component. A finite set of periodic orbits of the diffeomorphism gives rise to a braid in the mapping torus. We show that under some nondegeneracy hypotheses, the isotopy classes of braids that arise from finite sets of periodic orbits are stable under Hamiltonian perturbations that are small with respect to the Hofer metric. A corollary is that for a Hamiltonian isotopy class of such maps, the topological entropy is lower semicontinuous with respect to the Hofer metric. This extends  results of Alves-Meiwes for braids arising from finite sets of fixed points of Hamiltonian surface diffeomorphisms. 
\end{abstract}

\setcounter{tocdepth}{2}

\section{Introduction}
\label{sec:intro}

Given a vector field on a three-manifold, one can consider the knot types of its periodic orbits, and the link types of finite sets of periodic orbits. How do these knot and link types depend on the vector field? In this paper we consider the following special case of this question.

\subsection{Statement of the main result}
\label{sec:statement}

Let $\Sigma$ be a compact
surface, possibly with boundary, and let $\omega$ be a symplectic (area) form on $\Sigma$. Let $\phi:\Sigma\to\Sigma$ be an area-preserving diffeomorphism. A basic dynamical question is to understand the periodic orbits of $\phi$ and their properties. Our convention is that a {\em periodic orbit\/} of $\phi$ is a cycle of $k$ points $x_1,\ldots,x_k\in\Sigma$ for some positive integer $k$, called the {\em period\/}, such that $\phi(x_i)=x_{i+1\mod k}$ for each $i$. We say that the periodic orbit is {\em simple\/} if the points $x_1,\ldots, x_k$ are distinct; otherwise we say that it is {\em multiply covered\/}. Below, all periodic orbits are understood to be simple unless stated otherwise.

In this paper we are concerned with topological properties of the periodic orbits, namely braid types of periodic orbits or finite sets of periodic orbits. To be more precise, recall that the {\em mapping torus\/} of $\phi$ is the three-manifold $Y_\phi$ defined by
\[
Y_\phi = [0,1]\times\Sigma/\sim,\quad (1,x)\sim(0,\phi(x)).
\]
This is a fiber bundle over $S^1=\R/\Z$ with fiber $\Sigma$. We denote the $S^1$ coordinate by $t$. There is a canonical ``Reeb'' vector field $\partial_t$ on $Y_\phi$ which increases the $S^1$ coordinate. A simple periodic orbit of the map $\phi$ with period $k$ is equivalent to a simple (i.e.\ embedded) periodic orbit of the vector field $\partial_t$ whose projection to $S^1$ has degree $k$.

\begin{definition}
A {\em braid\/} in the mapping torus $Y_\phi$ is a link $\zeta$ in $Y_\phi$ which is transverse to the fibers of $Y_\phi\to S^1$. Two braids $\zeta_0,\zeta_1$ in $Y_\phi$ are {\em isotopic\/} if they are the endpoints of a smooth one-parameter family of braids $\{\zeta_t\}|_{t\in[0,1]}$. We denote the isotopy class of the braid $\zeta$ by $[\zeta]$.
\end{definition}

If $\alpha$ is a finite set of periodic orbits of the map $\phi$, regarded as periodic orbits of the vector field $\partial_t$, then $\alpha$ gives rise to a braid in $Y_\phi$, which we denote by $\zeta(\alpha)$. We say that the {\em braid type\/} of the set of periodic orbits $\alpha$ is the isotopy class $[\zeta(\alpha)]$.

\begin{definition}
Let $\mc{B}(\phi)$ denote the set of isotopy classes of braids $[\zeta(\alpha)]$ in $Y_\phi$ that arise from finite sets $\alpha$ of periodic orbits of $\phi$.
\end{definition}

We are interested in how the set $\mc{B}(\phi)$ of braid types behaves under Hamiltonian isotopy of $\phi$. To formulate this question more precisely, it is convenient to describe a Hamiltonian isotopy as in \cite[\S3]{pfh4}, as follows. We start with a smooth function
\[
H: Y_\phi \longrightarrow \R.
\]
Under the quotient map $[0,1]\times\Sigma\to Y_\phi$, the function $H$ pulls back to a function
\[
\widetilde{H}:[0,1]\times\Sigma\to\R
\]
satisfying $\widetilde{H}(1,x) = \widetilde{H}(0,\phi(x))$. For $t\in[0,1]$, let $H_t=\widetilde{H}(t,\cdot):\Sigma\to\R$. Let $X_{H_t}$ denote the associated Hamiltonian vector field on $\Sigma$, defined by the convention
\begin{equation}
\label{eqn:Hamiltonianconvention}
\omega(X_{H_t},\cdot)=dH_t.
\end{equation}
We always assume that $H_t$ is locally constant near $\partial\Sigma$ for each $t$, so that $X_{H_t}$ vanishes near $\partial\Sigma$. Let $\{\varphi_t:\Sigma\to\Sigma\}_{t\in[0,1]}$ denote the Hamiltonian isotopy defined by $\varphi_0=\op{id}_\Sigma$ and $\partial_t\varphi_t=X_{H_t}\circ\varphi_t$. We define $\phi_H=\phi\circ\varphi_1$.

There is a diffeomorphism of mapping tori
\begin{equation}
\label{eqn:fh}
f_H:Y_\phi\stackrel{\simeq}{\longrightarrow} Y_{\phi_H}
\end{equation}
induced by the diffeomorphism of $[0,1]\times\Sigma$ sending
\[
(t,x) \longmapsto (t,\varphi_t^{-1}(x)).
\]
One can now ask: How does the set of braid types $\mc{B}(\phi)$ relate to the set of braid types $\mc{B}(\phi_H)$ under the diffeomorphism $f_H$?

We will consider Hamiltonian perturbations for which $H$ is sufficiently small with respect to the following norm.

\begin{definition}
\label{def:Hoferenergy}
Given $H:Y_\phi\to\R$, define
\[
\|H\| = \max_{Y_\phi} H - \min_{Y_\phi} H.
\]
\end{definition}

\begin{remark}
It is more standard to measure the size of $H$ as
\begin{equation}
\label{eqn:hofer}
\|H\|' = \int_{0}^1\left(\max_{\Sigma}H_t - \min_{\Sigma} H_t\right)dt.
\end{equation}
For $\psi$ Hamiltonian isotopic to $\phi$, one then defines the {\em Hofer distance\/} by
\begin{equation}
\label{eqn:Hofer}
d_{\op{Hofer}}(\phi,\psi)=\inf\{\|H\|'\mid \phi_H=\psi\}.
\end{equation}
It is shown in \cite{hofer} that $d_{\op{Hofer}}$ is a metric on the set of maps in a given Hamiltonian isotopy class; the nontrivial part of this statement is that $d_{\op{Hofer}}(\phi,\psi)=0$ implies that $\phi=\psi$. Note that one obtains the same distance if one replaces $\|H\|'$ by $\|H\|$ in \eqref{eqn:Hofer}, because one can reparametrize a Hamiltonian isotopy in time to level out the integrand in \eqref{eqn:hofer}.
\end{remark}

To state the main result, we need to introduce some hypotheses.

\begin{definition}
Let $\phi:\Sigma\to\Sigma$ be an area-preserving diffeomorphism. We say that $\phi$ is {\em boundary-admissible\/} if:
\begin{itemize}
\item
$\phi$ sends each component of $\partial\Sigma$ to itself.
\item
For each component of $\partial\Sigma$, there is an irrational number $\theta$ and a neighborhood of the component identified with $(-\epsilon,0]\times (\R/\Z)$ for some $\epsilon>0$, on which
\begin{equation}
\label{eqn:phiboundary}
\phi(x,y) = (x,y+\theta).
\end{equation}
\end{itemize}
\end{definition}

\begin{definition}
Given a periodic orbit $\gamma=(x_1,\ldots,x_k)$ of $\phi$, consider the ``linearized return map''
\begin{equation}
\label{eqn:linretmap}
d\phi^k:T_{x_i}\Sigma \longrightarrow T_{x_i}\Sigma.
\end{equation}
We say that the periodic orbit $\gamma$ is {\em nondegenerate\/} if $1-d\phi^k:T_{x_i}\Sigma\to T_{x_i}\Sigma$ is invertible for each $i$ (or equivalently for any $i$). We say that the map $\phi$ is {\em nondegenerate\/} if all of its periodic orbits (including multiple covers) are nondegenerate.
\end{definition}

A standard argument shows that within any Hamiltonian isotopy class of boundary-admissible area-preserving surface diffeomorphisms, a $C^\infty$-generic diffeomorphism is nondegenerate.

We can now state the main result. Roughly speaking, this asserts that nondegenerate braid types of boundary-admissible area-preserving diffeomorphisms are stable under Hamiltonian isotopies that are small with respect to the Hofer metric.

\begin{theorem}
\label{thm:main}
Let $\Sigma$ be a compact surface with a symplectic form $\omega$. Let $\phi:\Sigma\to\Sigma$ be a boundary-admissible area-preserving diffeomorphism. Let $\alpha$ be a finite set of nondegenerate simple periodic orbits of $\phi$. Then there exists $\delta>0$ such that the following holds: Let $H:Y_\phi\to\R$ be a smooth function such that $H_t$ is locally constant near $\partial\Sigma$ for each $t$, the map $\phi_H$ is nondegenerate, and $\|H\|<\delta$. Then
\begin{equation}
\label{eqn:conclusion}
f_H([\zeta(\alpha)]) \in \mc{B}(\phi_H).
\end{equation}
\end{theorem}

\begin{remark}[classical fixed point theory]
\label{rem:classical}
If we replaced the condition $\|H\|<\delta$ in Theorem~\ref{thm:main} with the condition that $\phi_H$ is sufficiently $C^0$-close to $\phi$, then the conclusion \eqref{eqn:conclusion} would follow from classical fixed point theory. However proving Theorem~\ref{thm:main} is harder than this, because the assumption that $\|H\|$ is small does not imply that $\phi_H$ is $C^0$-close to $\phi$, since by equation \eqref{eqn:Hamiltonianconvention}, the isotopy from $\phi$ to $\phi_H$ depends on the first derivatives of $H$.
\end{remark}

\begin{remark}[hypotheses]
The hypothesis in Theorem~\ref{thm:main} that the periodic orbits in $\alpha$ are nondegenerate cannot be dropped completely; otherwise there could be a degenerate orbit coming from a birth-death bifurcation which can be destroyed by an arbitrarily $C^\infty$-small perturbation. (However extension to some degenerate cases is possible; see \cite{khanevsky}.) If we drop the hypothesis that $\phi_H$ is nondegenerate, then the most we can get (by taking a limit of nondegenerate Hamiltonian perturbations of $\phi_H$) is a ``weighted braid'' in $Y_{\phi_H}$; cf.\ \cite[\S3.1]{cghhl}. The hypothesis that $\phi$ is boundary-admissible can probably be weakened. We do not know what can be said about braid stability (beyond classical fixed point theory as in Remark~\ref{rem:classical}) for non-Hamiltonian isotopies, or more generally for surface diffeomorphisms that are not area-preserving.
\end{remark}

\begin{remark}[previous work]
\label{rem:am}
Alves-Meiwes \cite[Thm.\ 2.5]{am} proved a result similar to Theorem~\ref{thm:main} in which $\Sigma$ is closed or a disk, $\phi$ is Hamiltonian isotopic to the identity, and the periodic orbits in $\alpha$ are fixed points. Their proof uses holomorphic cylinders, provided by continuation maps on Hamiltonian Floer theory in an action window. Some precedents for this, detecting periodic orbits but not the braid type, are the work of Polterovich-Shelukhin \cite{ps} proving stability of barcodes in Hamiltonian Floer homology under the Hofer metric, and the earlier paper of Cieliebak-Floer-Hofer-Wysocki \cite{cfhw} establishing $C^0$-stability of the action spectrum of fillable contact manifolds.
\end{remark}

\begin{remark}[method of proof]
The proof of Theorem~\ref{thm:main} will find the required braid of periodic orbits of $\phi_H$, and the required isotopy with $\zeta(\alpha)$, also using holomorphic cylinders.
We will work directly with the holomorphic curves we need, without placing the discussion into the framework of maps on Floer theory. The idea is to consider holomorphic cylinders in symplectic cobordisms between $Y_\phi$ and itself. We start with ``trivial cylinders'' in the ``trivial cobordism'', and deform this cobordism via a kind of ``neck stretching'' into a composition of cobordisms between $Y_\phi$ and $Y_{\phi_H}$. The holomorphic cylinders will then break along sets of periodic orbits of $\phi_H$ in $Y_{\phi_H}$. The main issue is to avoid multiply covered periodic orbits of $\phi_H$. We will analyze how breaking can occur using some topological facts about holomorphic curves that enter into the construction of periodic Floer homology (PFH) of area-preserving surface diffeomorphisms, and its sister theory, embedded contact homology (ECH) of contact three-manifolds; see e.g.\ \cite{bn}. A key ingredient is to use ECH partition conditions to analyze the multiplicities of periodic orbits that arise in breakings. A parity argument will show that at least one such breaking gives us a set of simple periodic orbits of $\phi_H$. The holomorphic cylinders will then give the desired isotopy of braids.
\end{remark}

Theorem~\ref{thm:main} has the following application to the topological entropy $h_{\op{top}}$ of area-preserving surface diffeomorphisms; see e.g.\ \cite[\S1.4]{am} for a review of topological entropy.

\begin{corollary}
The topological entropy of boundary-admissible area-preserving surface diffeomorphisms in a given Hamiltonian isotopy class is lower semicontinuous with respect to the Hofer metric.
\end{corollary}

\begin{proof}
Let $\phi$ be a boundary-admissible area-preserving surface diffeomorphism. Given $\epsilon>0$, we need to show that there exists $\delta>0$ such that that if $H:Y_\phi\to\R$ is a Hamiltonian such that $H_t$ is locally constant near $\partial \Sigma$ for each $t$ and $\|H\|<\delta$, then
\begin{equation}
\label{eqn:htop0}
h_{\op{top}}(\phi_H) > h_{\op{top}}(\phi) - \epsilon.
\end{equation}
We can assume that $h_{\op{top}}(\phi)>0$, as otherwise \eqref{eqn:htop0} holds automatically.

Let $\mc{B}_h(\phi)$ denote the set of braid types in $\mc{B}(\phi)$ that can be represented by a finite set of hyperbolic (in particular nondegenerate) periodic orbits; see \S\ref{sec:orbitsets} below for the definition. It is shown in \cite[\S1.4]{am}, using methods from \cite{bowen,fh,katok}, that under our assumption that $h_{\op{top}}(\phi)>0$, we have
\begin{equation}
\label{eqn:htop}
\sup_{\zeta\in\mc{B}_h(\phi)} h_{\op{top}}(\zeta) = h_{\op{top}}(\phi) = \sup_{\zeta\in\mc{B}(\phi)} h_{\op{top}}(\zeta).
\end{equation}
Here if $\zeta\in\mc{B}(\phi)$, the number $h_{\op{top}}(\zeta)$ is defined in terms of the growth rate of the action of $\phi^n$ on $\pi_1$ of the complement in $\Sigma$ of the set of periodic points corresponding to $\zeta$. (In fact, a version of \eqref{eqn:htop} holds for any diffeomorphism of a compact surface.)

As noted in \cite[\S1.4]{am}, it follows from the combination of \cite{newhouse,yomdin} that the topological entropy of area-preserving diffeomorphisms is continuous with respect to the $C^\infty$ topology. Consequently, in proving \eqref{eqn:htop0}, we can restrict attention to the case when $\phi_H$ is nondegenerate. By the first part of \eqref{eqn:htop}, we can find $\zeta\in\mc{B}_h(\phi)$ such that
\begin{equation}
\label{eqn:htop1}
h_{\op{top}}(\zeta) > h_{\op{top}}(\phi) - \epsilon.
\end{equation}

Now Theorem~\ref{thm:main} provides a number $\delta>0$ such that if $H_t$ is locally constant near $\partial\Sigma$ for each $t$ and $\|H\|<\delta$ and $\phi_H$ is nondegenerate, then there exists $\zeta'\in\mc{B}(\phi_H)$ which is isotopic to $f_H(\zeta)$. It follows from this isotopy that
\begin{equation}
\label{eqn:htop2}
h_{\op{top}}(\zeta') = h_{\op{top}}(\zeta).
\end{equation}
By the second part of \eqref{eqn:htop}, we have
\begin{equation}
\label{eqn:htop3}
h_{\op{top}}(\phi_H) \ge h_{\op{top}}(\zeta').
\end{equation}
Combining \eqref{eqn:htop1}, \eqref{eqn:htop2}, and \eqref{eqn:htop3} proves \eqref{eqn:htop0}.
\end{proof}

\begin{remark}
One can ask if Theorem~\ref{thm:main} has an analogue for contact three-manifolds. For example, suppose $Y$ is a closed three-manifold, $\lambda$ is a contact form on $Y$ with associated contact structure $\xi = \op{Ker}(\lambda)$, and $\zeta\subset Y$ is a finite union of nondegenerate simple Reeb orbits of $\lambda$. Does there exist $\epsilon>0$ such that if $f:Y\to\R$ satisfies $|f|<\epsilon$, and the contact form $e^f\lambda$ is nondegenerate, then this contact form has a finite union of simple Reeb orbits which is isotopic to $\zeta$ as a transverse link in $(Y,\xi)$?
\end{remark}

\paragraph{Funding.} This work was partially supported by the National Science Foundation [DMS-2005437].

\paragraph{Acknowledgments.} The author is grateful to Marcelo Alves for a number of helpful comments on a draft of this paper and explanations of key points in \cite{am}. The author thanks Julian Chaidez and Yuan Yao for pointing out an error in a previous arXiv version of this paper. The author thanks the anonymous referee for suggesting several valuable corrections and clarifications.


\section{Preliminaries}

Fix a boundary-admissible area-preserving diffeomorphism $\phi$ of $(\Sigma,\omega)$. We now introduce some notions about holomorphic curves in $\R\times Y_\phi$ and related cobordisms that we will need.  Many of these notions are ingredients in the definition of periodic Floer homology (PFH), although we will not actually need to use PFH here. More details about these PFH ingredients may be found in the lecture notes \cite{bn} (which treat the analogous case of ECH of contact three-manifolds), and the paper \cite{ir} (which treats stable Hamiltonian structures, including both mapping tori of area-preserving surface diffeomorphisms and contact three-manifolds).

\subsection{Orbit sets and action differences}
\label{sec:orbitsets}

\begin{definition}
\label{def:orbitset}
An {\em orbit set\/} is a finite set of pairs $\alpha=\{(\alpha_i,m_i)\}$ where the $\alpha_i$ are distinct simple periodic orbits of $\phi$, and the $m_i$ are positive integers. We define the homology class of $\alpha$ by
\[
[\alpha]=\sum_im_i[\alpha_i]\in H_1(Y_\phi).
\]
We define the {\em degree\/} of $\alpha$ by $d=[\alpha]\cdot[\Sigma]\in\Z_{\ge 0}$, where $[\Sigma]$ denotes the homology class of a fiber of $Y_\phi\to S^1$. Equivalently, $d$ is the sum over $i$ of $m_i$ times the period of $\alpha_i$.
\end{definition}

\begin{definition}
\label{def:chains}
If $\alpha$ and $\beta$ are orbit sets with $[\alpha]=[\beta]\in H_1(Y_\phi)$, let $H_2(Y_\phi,\alpha,\beta)$ denote the set of $2$-chains $Z$ in $Y_\phi$ with $\partial Z = \alpha - \beta$, modulo boundaries of $3$-chains. The set $H_2(Y_\phi,\alpha,\beta)$ is an affine space over $H_2(Y_\phi)$.
\end{definition}

Below, let $\omega_\phi$ denote the unique $2$-form on the mapping torus $Y_\phi$ that restricts to the symplectic form $\omega$ on $\Sigma$ on each fiber of the projection $Y_\phi\to S^1$, and that annihilates the vector field $\partial_t$. Note that $d\omega_\phi=0$.

\begin{definition}
\label{def:ad}
If $\alpha$ and $\beta$ are orbit sets with $[\alpha]=[\beta]\in H_1(Y_\phi)$, and if $Z\in H_2(Y_\phi,\alpha,\beta)$, define the {\em action difference\/}
\[
\mc{A}(Z) = \int_Z\omega_\phi.
\]
\end{definition}

\begin{definition}
\label{def:simple}
The orbit set $\alpha=\{(\alpha_i,m_i)\}$ is {\em simple\/} if $m_i=1$ for all $i$. The orbit set $\alpha$ is {\em nondegenerate\/} if each periodic orbit $\alpha_i$, along with its multiple covers up to multiplicity $m_i$, is nondegenerate.
\end{definition}

\begin{definition}
Let $\gamma=(x_1,\ldots,x_k)$ be a nondegenerate periodic orbit of $\phi$. For each $i=1,\ldots,k$ we have the linearized return map \eqref{eqn:linretmap}, which is a symplectic map whose eigenvalues do not depend on $i$. We say that the periodic orbit $\gamma$ is {\em elliptic\/} if the eigenvalues are on the unit circle, {\em positive hyperbolic\/} if the eigenvalues are positive, and {\em negative hyperbolic\/} if the eigenvalues are negative.
\end{definition}

\begin{remark}
\label{rem:generators}
In the definition of PFH (which we will not need here), one ordinarily assumes that $\phi$ is nondegenerate\footnote{One can also generalize to Morse-Bott settings; for Morse-Bott ECH see \cite{cghy,nw,yuan}.}. The generators of the chain complex are then orbit sets $\alpha=\{(\alpha_i,m_i)\}$ such that $m_i=1$ whenever the orbit $\alpha_i$ is hyperbolic.
\end{remark}

\subsection{Holomorphic curves in $\R\times Y_\phi$}
\label{sec:curves}

Let $E\to Y_\phi$ denote the vertical tangent bundle of $Y_\phi\to S^1$.

\begin{definition}
\label{def:Jphi}
Let $\mc{J}(\phi)$ denote the set of almost complex structures $J$ on $\R\times Y_\phi$ such that:
\begin{itemize}
\item $J(\partial_s)=\partial_t$, where $s$ denotes the $\R$ coordinate on $\R\times Y$.
\item $J$ is invariant under translation in the $\R$ direction.
\item $J$ maps $E$ to itself, rotating positively with respect to $\omega$.
\item Near the boundary of $\R\times Y_\phi$, in the coordinates of \eqref{eqn:phiboundary}, we have
\begin{equation}
\label{eqn:Jboundary}
J\partial_x = \partial_y.
\end{equation}
\end{itemize}
\end{definition}

Fix $J\in\mc{J}(\phi)$. We consider $J$-holomorphic curves in $\R\times Y_\phi$, namely (nonconstant) holomorphic maps
\[
u: (C,j) \longrightarrow (\R\times Y_\phi,J),
\]
where $C$ is a (connected) compact Riemann surface with a finite number of punctures, modulo reparametrization by biholomorphisms of the domain. We assume that each puncture of $C$ is positively or negatively asymptotic to a (not necessarily simple) nondegenerate periodic orbit of $\phi$. Here if $\gamma$ is a periodic orbit of period $k$, regarded as a loop in $Y_\phi$, then ``positively asymptotic to $\gamma$'' means that there is a number $R>>0$, and a neighborhood of the puncture in $C$, holomorphically identified with $[R,\infty)\times S^1$, on which
\[
u(s,t) = (ks, \eta(s,t)),
\]
where $\lim_{s\to\infty}\eta(s,\cdot)$ is a parametrization of $\gamma$. ``Negatively asymptotic to $\gamma$'' is defined analogously with $s\to -\infty$.

Any such $J$-holomorphic curve $u$ is either {\em somewhere injective\/} --- meaning  that there exists $z\in C$ such that $du_z$ is injective and $u^{-1}(u(z))=\{z\}$ --- or {\em multiply covered\/}, meaning that $u$ factors through a branched cover $(C,j)\to (C',j')$ of degree greater than $1$. When $u$ is somewhere injective, we identify $u$ with its image in $\R\times Y_\phi$, which by abuse of notation we denote by $C$. 

\begin{definition}
\label{def:regular}
An almost complex structure $J\in\mc{J}(\phi)$ is {\em regular\/} if every somewhere injective $J$-holomorphic curve $C$ as above that is not a fiber of the projection $\R\times Y_\phi\to \R\times S^1$ is cut out transversely (see e.g.\ \cite[Def.\ 7.14]{wendl} for a precise definition of this). 
\end{definition}

Standard transversality arguments in \cite{dragnev} or \cite[Thm.\ 8.1]{wendl} show that if $J\in\mc{J}(\phi)$ is generic, then $J$ is regular in the above sense. In this case, for every $J$-holomorphic curve $C$ that is not a fiber and whose ends are asymptotic to nondegenerate periodic orbits, the moduli space of $J$-holomorphic curves is a manifold near $C$ whose dimension is the Fredholm index $\op{ind}(C)$. Although we do not need the detailed formula here, we note that the Fredholm index is given by
\begin{equation}
\label{eqn:Fredholmindex}
\op{ind}(C) = -\chi(C) + 2c_\tau(C) + \sum_i\op{CZ}_\tau(\gamma_i^+) - \sum_j\op{CZ}_\tau(\gamma_j^-).
\end{equation}
Here $\gamma_i^+$ are the (possibly multiply covered) periodic orbits to which ends of $C$ are positively asymptotic, and $\gamma_j^-$ are the periodic orbits to which ends of $C$ are negatively asymptotic; $\tau$ is a homotopy class of symplectic trivialization of $E$ over the periodic orbits $\gamma_i^+$ and $\gamma_j^-$; and $c_\tau(C)\in\Z$ denotes the relative first Chern class of $E$ over $C$ with respect to the boundary trivialization $\tau$. Finally, if $\gamma$ is a nondegenerate Reeb orbit and $\tau$ is a trivialization of $E$ over $\gamma$, then $\op{CZ}_\tau(\gamma)\in\Z$ denotes the Conley-Zehnder index of $\gamma$ with respect to $\tau$. See \cite[\S3.2]{bn} for detailed definitions of these notions.

\begin{remark}
\label{rem:fibers}
If $\Sigma$ is a closed surface of genus $g>0$, then the fibers of the projection $\R\times Y_\phi\to\R\times S^1$ are $J$-holomorphic curves which are not cut out transversely, because they have Fredholm index $2-2g\le 0$ but come in a 2-dimensional moduli space. If $\Sigma=S^2$, then the fibers are cut out transversely, and they have Fredholm index $2$.
\end{remark}

\begin{remark}
\label{rem:cylinders}
Suppose $C$ is a somewhere injective $J$-holomorphic curve which is not a fiber. It follows from the conditions on $J$ in Definition~\ref{def:Jphi} that the restriction to $C$ of the projection $\R\times Y_\phi\to\R\times S^1$ is a branched cover
\begin{equation}
\label{eqn:branchedcover}
C \longrightarrow \R\times S^1.
\end{equation}
The case we will need for the proof of the main theorem is where $C$ is a cylinder. In this case the map \eqref{eqn:branchedcover} cannot have any branch points, so $C$ is transverse to the fibers of $\R\times Y_\phi\to\R\times S^1$.
\end{remark}

\subsection{Holomorphic currents in $\R\times Y_\phi$}
\label{sec:currents}

\begin{definition}
As in \cite[\S3.1]{bn}, we define a {\em $J$-holomorphic current\/} to be a finite sum $\mc{C} = \sum_id_iC_i$, where the $C_i$ are distinct somewhere injective $J$-holomorphic curves as in \S\ref{sec:curves}, and the $d_i$ are positive integers. We say that the $J$-holomorphic current $\mc{C}$ {\em does not have any multiply covered components\/} if each $d_i=1$. \end{definition}

If $\mc{C}$ does not have any multiply covered components, then we can regard $\mc{C}$ as a $J$-holomorphic curve with disconnected domain, and by abuse of notation we identify $\mc{C}$ with the set $C=\union_i C_i\subset\R\times Y_\phi$. It follows from results in \cite{siefring}, as explained in \cite[\S3.3]{bn}, that if the ends of $C$ are asymptotic to nondegenerate periodic orbits, then the set $C$ is a submanifold of $\R\times Y_\phi$, except possibly for finitely many singular points.

\begin{definition}
If $\alpha$ and $\beta$ are nondegenerate orbit sets with $[\alpha]=[\beta]\in H_1(Y_\phi)$, and if $J\in\mc{J}(\phi)$, let $\mc{M}^J(\alpha,\beta)$ denote the set of $J$-holomorphic currents in $\R\times Y_\phi$ that are asymptotic as currents to $\alpha$ as $s\to+\infty$ and to $\beta$ as $s\to-\infty$. A holomorphic current $\mc{C}\in\mc{M}^J(\alpha,\beta)$ has a well-defined relative homology class $[\mc{C}]\in H_2(Y_\phi,\alpha,\beta)$. Given $Z\in H_2(Y_\phi,\alpha,\beta)$, let $\mc{M}^J(\alpha,\beta,Z)$ denote the set of holomorphic currents $\mc{C}\in \mc{M}^J(\alpha,\beta)$ such that $[\mc{C}]=Z$.
\end{definition}

\begin{example}
\label{ex:trivialcylinders}
If $\alpha=\{(\alpha_i,m_i)\}$ is a nondegenerate orbit set, let $\R\times\alpha$ denote the sum over $i$ of the ``trivial cylinder'' $\R\times\alpha_i$ with multiplicity $m_i$. Then for any $J\in\mc{J}(\phi)$, we have $\R\times\alpha\in\mc{M}^J(\alpha,\alpha,0)$.
\end{example}

\begin{lemma}
\label{lem:actionfiltration}
If $J\in\mc{J}(\phi)$ and $\mc{C}\in \mc{M}^J(\alpha,\beta)$, then $\mc{A}([\mc{C}])\ge 0$, with equality if and only if $\alpha=\beta$ and $\mc{C}=\R\times\alpha$.
\end{lemma}

\begin{proof}
It follows from Definition~\ref{def:Jphi} that $\omega_\phi$ is pointwise nonnegative on any somewhere injective $J$-holomorphic curve $C$, with equality at a nonsingular point $p$ if and only if $T_pC=\op{span}(\partial_s,\partial_t)$.
\end{proof}

\subsection{Isolation of nondegenerate orbit sets}

\begin{definition}
\label{def:isolation}
Let $\alpha$ be a nondegenerate orbit set, let $J\in\mc{J}(\phi)$, and let $\epsilon>0$.
\begin{description}
\item{(a)}
Suppose that $\phi$ is nondegenerate.
We say that $\alpha$ is {\em $\epsilon$-isolated for $J$\/} if for every $\mc{C}\in\mc{M}^J(\alpha,\beta)\cup\mc{M}^J(\beta,\alpha)$, where either $\alpha\neq\beta$, or $\alpha=\beta$ and $[\mc{C}]\neq 0$, we have $\mc{A}([\mc{C}])> \epsilon$.
\item{(b)}
We say that $\alpha$ is {\em strongly $\epsilon$-isolated for $J$\/} if the above holds for $(\phi',J')$ whenever $\phi'$ is a nondegenerate $C^\infty$-small perturbation of $\phi$ for which $\alpha$ is still an orbit set, and $J'\in\mc{J}(\phi')$ is $C^\infty$-close to $J$.
\end{description}
\end{definition}

\begin{lemma}
\label{lem:isolation}
Let $\alpha$ be a nondegenerate orbit set and let $J\in\mc{J}(\phi)$. Then there exists $\epsilon>0$ such that $\alpha$ is strongly $\epsilon$-isolated for $J$.
\end{lemma}

\begin{proof}
Suppose to get a contradiction that there exists a sequence of pairs $(\phi_n,J_n)$ as in Definition~\ref{def:isolation}(b) converging in $C^\infty$ to $(\phi,J)$, orbit sets $\beta^n$ of $\phi_n$, and $J_n$-holomorphic currents $\mc{C}_n\in \mc{M}^{J_n}(\alpha,\beta^n)\cup \mc{M}^{J_n}(\beta^n,\alpha)$ where $[\mc{C}_n]\neq 0$ if $\alpha=\beta^n$, such that $\lim_{n\to\infty}\mc{A}([\mc{C}_n])=0$.

Write $\alpha=\{(\alpha_i,m_i)\}$, and let $k_i$ denote the period of $\alpha_i$. Since $\alpha$ is nondegenerate, it follows that for each $i$, there exists a tubular neighborhood $N_i$ of $\alpha_i$ in $Y_\phi$ such that if $n$ is sufficiently large, then the only periodic orbits of $\phi_n$ with period $\le k_im_i$ contained in $N_i$ are multiple covers of $\alpha_i$. 
 We can arrange that the $N_i$ are disjoint, by shrinking them if necessary. Let $N=\coprod_iN_i$. 
Since $N$ deformation retracts onto $\coprod_i\alpha_i$, it follows that if $n$ is sufficiently large as above then $\mc{C}_n\cap \partial N \neq \emptyset$. For each such $n$ choose a point $y_n\in\mc{C}_n\cap \partial N$.

Since $\partial N$ is compact, we can pass to a subsequence so that the points $y_n$ converge to a point $y_\infty\in\partial N$. As in \cite[\S3]{twoorbits}, one can use the Gromov compactness for currents in \cite[Prop.\ 3.3]{taubescurrents} or \cite[Prop.\ 1.9]{dw} to find a $J$-holomorphic current\footnote{Note that if $\phi$ is degenerate, then $\mathcal{C}_\infty$ does not necessarily have all ends asymptotic to periodic orbits; see \cite{siefringmultiple} for one way this can fail.} $\mc{C}_\infty$ containing $y_\infty$ with $\lim\inf_{n\to\infty}\mc{A}([\mc{C}_n])\ge \int_{\mathcal{C}_\infty}\omega_\phi>0$. This is the desired contradiction.
\end{proof}

\begin{remark}
In the special case where $\phi$ is nondegenerate and rational, we can obtain a more explicit constant $\epsilon>0$, such that $\alpha$ is $\epsilon$-isolated for all $J$. Here, following \cite[\S1.1]{pfh4}, we say that $\phi$ is {\em rational\/} if the cohomology class $[\omega_\phi]\in H^2(Y_\phi;\R)$ is a real multiple of the image of an integral class.

To do so, let $\alpha$ be an orbit set. Since $\Sigma$ is compact and $\phi$ is nondegenerate, there are only finitely many periodic orbits with any given period, hence only finitely many orbit sets with any given degree, hence only finitely many orbit sets $\beta$ with $[\alpha]=[\beta]$. Since $\phi$ is rational, for each such $\beta$, the set of action differences $\mc{A}(Z)$ for $Z\in H_2(Y_\phi,\alpha,\beta)$ is discrete. There is then a smallest positive value of $|\mc{A}(Z)|$ for $Z\in H_2(Y_\phi,\alpha,\beta)$, and we can take $\epsilon$ to be any positive number smaller than this smallest positive value. 
\end{remark}

\subsection{Partition conditions}
\label{sec:partitions}

Suppose that $\alpha=\{(\alpha_i,m_i)\}$ and $\beta=\{(\beta_j,n_j)\}$ are nondegenerate orbit sets and that $C\in\mc{M}^J(\alpha,\beta)$ does not have any multiply covered components. For each $i$, the curve $C$ has ends positively asymptotic to covers of the periodic orbit $\alpha_i$ with total multiplicity $m_i$. These multiplicities give a partition of $m_i$, which we denote by $p_i^+(C)$. Likewise, for each $j$, the curve $C$ has ends negatively asymptotic to covers of $\beta_j$ with total multiplicity $n_j$, and these multiplicities give a partition of $n_j$ which we denote by $p_j^-(C)$.

In the ECH index inequality which we will review in \S\ref{sec:ECHindex} below, equality holds only if these partitions satisfy certain conditions, which we now review. These partition conditions originally appeared in \cite[\S4]{pfh2}; for more about them, see e.g.\ \cite[\S3.9]{bn}.

\begin{definition}
Let $\gamma$ be a nondegenerate simple periodic orbit of $\phi$ and let $m$ be a positive integer. Define two partitions of $m$, the {\em positive partition\/} $p_\gamma^+(m)$ and the {\em negative partition\/} $p_\gamma^-(m)$, as follows.
\begin{itemize}
\item
If $\gamma$ is positive hyperbolic, then
\begin{equation}
\label{eqn:php}
p_\gamma^+(m) = p_\gamma^-(m) = (1,\ldots,1).
\end{equation}
\item
If $\gamma$ is negative hyperbolic, then
\begin{equation}
\label{eqn:nhp}
p_\gamma^+(m) = p_\gamma^-(m) = 
\left\{
\begin{array}{cl} (2,\ldots,2), & \mbox{if $m$ is even},\\
(2,\ldots,2,1), & \mbox{if $m$ is odd}.
\end{array}
\right.
\end{equation}
\item
If $\gamma=(x_1,\ldots,x_k)$ is elliptic, let $\theta\in\R/\Z$ denote the rotation number of $\gamma$. (This means that one can choose a symplectomorphism $T_{x_i}\Sigma\simeq \R^2$ in which $d\phi^k_{x_i}$ is identified with rotation by angle $2\pi\theta$.) Let $\Lambda_\theta^+(m)$ denote the maximal polygonal path in the plane with vertices at lattice points which is the graph of a concave function $[0,m]\to\R$ and does not go above the line $y=\theta x$. Then $p_\gamma^+(m)$ consists of the horizontal displacements of the segments of $\Lambda_\theta^+(m)$ connecting consecutive lattice points. Likewise, let $\Lambda_\theta^-(m)$ denote the minimal polygonal path in the plane with vertices at lattice points which is the graph of a convex function $[0,m]\to\R$ and does not go below the line $y=\theta x$. Then $p_\gamma^-(m)$ consists of the horizontal displacements of the segments of $\Lambda_\theta^-(m)$ connecting consecutive lattice points.
\end{itemize}
\end{definition}

\begin{definition}
\label{def:pc}
Let $\alpha=\{(\alpha_i,m_i)\}$ and $\beta=\{(\beta_j,n_j)\}$ be nondegenerate orbit sets, and suppose $C\in\mc{M}^J(\alpha,\beta)$ has no multiply covered components. We say that $C$ {\em satisfies the ECH partition conditions\/} if $p_i^+(C) = p_{\alpha_i}^+(m_i)$  for each $i$, and $p_j^-(C) = p_{\beta_j}^-(n_j)$ for each $j$.
\end{definition}

In \S\ref{sec:key} will need the following fact, which is a special case of \cite[Ex.\ 3.13(c)]{bn}:

\begin{lemma}
\label{lem:partitionsdisjoint}
If $\gamma$ is elliptic and $m>1$ and the $m$-fold cover of $\gamma$ is nondegenerate, then $p_\gamma^+(m)$ and $p_\gamma^-(m)$ are disjoint.
\end{lemma}

\subsection{The ECH index}
\label{sec:ECHindex}

\begin{definition}
\label{def:I}
If $\alpha=\{(\alpha_i,m_i)\}$ and $\beta=\{(\beta_j,n_j)\}$ are nondegenerate orbit sets with $[\alpha]=[\beta]\in H_1(Y_\phi)$, and if $Z\in H_2(Y_\phi,\alpha,\beta)$, define the {\em ECH index\/} $I(\alpha,\beta,Z)\in\Z$ by
\begin{equation}
\label{eqn:ECHindex}
I(\alpha,\beta,Z) = c_\tau(Z) + Q_\tau(Z) + \sum_i\sum_{k=1}^{m_i}\op{CZ}_\tau(\alpha_i^k) - \sum_j\sum_{l=1}^{n_j}\op{CZ}_\tau(\beta_j^l).
\end{equation}
Here $\tau$ is a homotopy class of symplectic trivialization of $E$ over the periodic orbits $\alpha_i$ and $\beta_j$; $c_\tau(Z)$ is the relative first Chern class of $E$ over $Z$ with respect to the boundary trivialization $\tau$; $Q_\tau(Z)$ is the relative self-intersection number of $Z$ with respect to $\tau$; and $\op{CZ}_\tau(\gamma^k)$ denotes the Conley-Zehnder index of the $k$-fold multiple cover of $\gamma$ with respect to $\tau$. If $\mc{C}\in\mc{M}^J(\alpha,\beta)$, write $I(\mc{C})=I(\alpha,\beta,[\mc{C}])$.
\end{definition}

The details of the above notions are not needed here and can be found in \cite[\S3]{bn} or \cite[\S2]{ir}. We do need to know the following two properties of the ECH index. First, a basic property, proved in \cite[Prop.\ 1.6]{pfh2}, is that it is additive in the following sense: If $\gamma$ is another nondegenerate orbit set with $[\alpha]=[\beta]=[\gamma]$, and if $W\in H_2(Y_\phi,\beta,\gamma)$, then $Z+W\in H_2(Y_\phi,\alpha,\gamma)$, and we have
\begin{equation}
\label{eqn:additive}
I(\alpha,\gamma,Z+W) = I(\alpha,\beta,Z) + I(\beta,\gamma,W).
\end{equation}
Second, the following ``index inequality'', relating the ECH index \eqref{eqn:ECHindex} to the Fredholm index \eqref{eqn:Fredholmindex}, is a fundamental result in the foundations of PFH and ECH. It is a special case\footnote{The paper \cite{ir} assumes that all periodic orbits are nondegenerate; but all that is needed for the proof is that the orbit sets $\alpha$ and $\beta$ are nondegenerate.} of \cite[Thm.\ 4.15]{ir}.

\begin{proposition}
\label{prop:ie}
Let $\alpha$ and $\beta$ be nondegenerate orbit sets, and suppose that $C\in\mc{M}^J(\alpha,\beta)$ has no multiply covered components. Then
\[
\op{ind}(C)\le I(C).
\]
Equality holds only if $C$ is embedded and satisfies the ECH partition conditions in Definition~\ref{def:pc}.
\end{proposition}

Using Proposition~\ref{prop:ie}, one can deduce the following key result. The analogous statement for ECH is part of \cite[Prop.\ 3.7]{bn}, and the proof of the proposition below is the same.

\begin{proposition}
\label{prop:fundamental}
Suppose $J\in\mc{J}(\phi)$ is regular. Let $\alpha$ and $\beta$ be nondegenerate orbit sets with $[\alpha]=[\beta]\in H_1(Y_\phi)$, and let $\mc{C}\in \mc{M}^J(\alpha,\beta)$. If $\Sigma$ is a closed surface of genus $g>0$, assume that $\mathcal{C}$ does not contain any fibers as in Remark~\ref{rem:fibers}. Then:
\begin{description}
\item{(a)} $I(\mc{C})\ge 0$.
\item{(b)} $I(\mc{C})=0$ if and only if $\alpha=\beta$ and $\mc{C}=\R\times\alpha$; see Example~\ref{ex:trivialcylinders}.
\end{description}
\end{proposition}


\section{Cobordisms between Hamiltonian perturbations}
\label{sec:cobordisms}

Fix a boundary-admissible area-preserving diffeomorphism $\phi:\Sigma\to\Sigma$. We now study four-dimensional symplectic cobordisms between mapping tori of different Hamiltonian perturbations of $\phi$.

Fix smooth maps $H_-, H_+: Y_\phi\to\R$ with $H_- \le H_+$, and assume that $(H_\pm)_t$ is locally constant near $\partial\Sigma$ for each $t\in S^1$. 

\subsection{Cobordism setup}

Let $\pi:\R\times Y_\phi\to\R\times S^1$ denote the projection. If $H:\R\times Y_\phi\to\R$ is a smooth map, let $H_{s,t}$ denote the restriction of $H$ to the fiber $\pi^{-1}(s,t)\in\R\times Y_\phi$, and let $X_{H_{s,t}}$ denote the associated Hamiltonian vector field on the fiber.

\begin{definition}
\label{def:monotonehomotopy}
A {\em monotone homotopy\/} from $H_-$ to $H_+$ is a smooth map $H:\R\times Y_\phi\to\R$ such that:
\begin{description}
\item{(a)}
There exists $s_0>0$ such that for every $y\in Y_\phi$ we have
\[
H(s,y)=\left\{\begin{array}{cl} H_+(y), & s\ge s_0,\\
H_-(y), & s \le -s_0.
\end{array}\right.
\]
\item{(b)}
If $s_1\le s_2$ and $y\in Y_\phi$ then $H(s_1,y) \le H(s_2,y)$.
\item{(c)}
$H_{s,t}$ is locally constant near the boundary of each fiber.
\end{description}
\end{definition}

We now make the following extension of Definition~\ref{def:Jphi}.

\begin{definition}
\label{def:Jcobnew}
Let $\mathcal{J}(\phi,H_+,H_-)$ denote the set of almost complex structures $J$ on $\R\times Y_\phi$ such that:
\begin{description}
\item{(a)}
The almost complex structure $J$ sends the vertical tangent bundle $E$ to itself, rotating positively with respect to the symplectic form $\omega$.
\item{(b)}
Equation \eqref{eqn:Jboundary} holds near $\R$ cross the boundary of each $\Sigma$ fiber.
\item{(c)}
There is a monotone homotopy $H$ from $H_-$ to $H_+$ such that
\begin{equation}
\label{eqn:Floer}
J(\partial_s) = \partial_t + X_{H_{s,t}}.
\end{equation}
\end{description}
\end{definition}

\begin{remark}
As a special case, if $H_+=H_-=0$, then $\mathcal{J}(\phi,H_+,H_-)=\mathcal{J}(\phi)$.
\end{remark}

\begin{remark}
\label{rem:holoproj}
Suppose that $J\in \mathcal{J}(\phi,H_+,H_-)$. Then it follows from (a) and (c) in Definition~\ref{def:Jcobnew} that the projection $\pi: \R\times Y_\phi\to\R\times S^1$ is holomorphic, with respect to $J$ and the standard almost complex structure on $\R\times S^1$ sending $\partial_s\mapsto\partial_t$.
\end{remark}

Next, define diffeomorphisms
\[
g_{H_\pm} = \op{id}_\R\times f_{H_\pm}  :\R\times Y_\phi \stackrel{\simeq}{\longrightarrow} \R\times Y_{\phi_{H_\pm}}.
\]

\begin{remark}
\label{rem:extending}
Suppose that $J\in \mathcal{J}(\phi,H_+,H_-)$, and let $s_0\in\R$ be as in Definition~\ref{def:monotonehomotopy} for the associated monotone homotopy $H$. Then it follows from equation \eqref{eqn:Floer} that there are almost complex structures $J_\pm\in\mathcal{J}(\phi_{H_\pm})$ such that
\[
J = \left\{\begin{array}{cl}
g_{H_+}^*J_+, & s\ge s_0,\\
g_{H_-}^*J_-, & s\le -s_0.
\end{array}
\right.
\]
\end{remark}


\subsection{Holomorphic curves in cobordisms}

\begin{definition}
\label{def:curvescobnew}
Let $J\in\mc{J}(\phi,H_+,H_-)$. If $\alpha^+$ and $\alpha^-$ are nondegenerate orbit sets for $\phi_{H_+}$ and $\phi_{H_-}$ respectively, let $\mc{M}^J(\alpha^+,\alpha^-)$ denote the set of $J$-holomorphic currents $\mc{C}$ in $\R\times Y_\phi$ such that:
\begin{itemize}
\item $\mc{C}$ is asymptotic as a current to $g_{H_+}^{-1}(\R\times\alpha^+)$ as $s\to+\infty$.
\item $\mc{C}$ is asymptotic as a current to $g_{H_-}^{-1}(\R\times\alpha^-)$ as $s\to-\infty$.
\end{itemize}
\end{definition}

Note that for any $\mc{C}\in\M^J(\alpha^+,\alpha^-)$ the projection of $\mathcal{C}$ to $Y_\phi$ gives a well-defined relative homology class
\[
[\mc{C}]\in H_2(Y_\phi,f_{H_+}^{-1}(\alpha^+),f_{H_-}^{-1}(\alpha^-)),
\]
Given $Z\in H_2(Y_\phi,f_{H_+}^{-1}(\alpha^+),f_{H_-}^{-1}(\alpha^-))$, we define
\[
\M^J(\alpha^+,\alpha^-,Z) = \left\{\mc{C}\in\M^J(\alpha^+,\alpha^-) \mid [\mc{C}]=Z\right\}.
\]
We also define the action difference
\begin{equation}
\label{eqn:adcob}
\mathcal{A}(Z) = \int_Z\omega_\phi + \int_{f_{H_+}^{-1}(\alpha^+)}H_+dt - \int_{f_{H_-}^{-1}(\alpha^-)}H_-dt.
\end{equation}

\begin{remark}
\label{rem:ade}
In the special case where $H_+=H_-=H:Y_\phi\to\R$, we have $(f_H)_*Z\in H_2(Y_{\phi_H},\alpha^+,\alpha^-)$, and \eqref{eqn:adcob} reverts to the previous notion of action difference $\mathcal{A}((f_H)_*Z)$ as in Definition~\ref{def:ad}. This is because $f_H^*\omega_{\phi_H} = \omega_\phi + dH\wedge dt$.
\end{remark}

We have the following generalization of Lemma~\ref{lem:actionfiltration}.

\begin{lemma}
\label{lem:stokesnew}
If $\mc{C}\in\M^J(\alpha^+,\alpha^-)$, then the action difference
\begin{equation}
\label{eqn:stokesnew}
\mc{A}([\mc{C}]) \ge  0.
\end{equation}
\end{lemma}

\begin{proof}
This is a variant of a standard calculation in Hamiltonian Floer theory; see e.g.\ \cite[Prop.\ 11.1.2]{ad}.

By linearity, we may assume without loss of generality that $\mc{C}$ consists of a single somewhere injective $J$-holomorphic curve $u:C\to\R\times Y_\phi$. Let $z\in C$, and assume that $z$ is not one of the (at most finitely many) ramification points of the composition $\pi\circ u: C \to \R\times S^1$.  Then we can use the coordinates $s$ and $t$ on $\R\times S^1$ as local coordinates on $C$ near $z$ (where we assume that $t$ has been locally lifted from $\R/\Z$ to $\R$). Near $z$, the map $u$ has the form
\[
u(s,t) = (s,[(t,v(s,t))])
\]
where $v(s,t)\in\Sigma$. By equation \eqref{eqn:Floer}, $v$ satisfies the parametrized Floer equation
\begin{equation}
\label{eqn:pfe}
\partial_s v + J(\partial_tv - X_{H_{s,t}}) = 0.
\end{equation}
By condition (a) in Definition~\ref{def:Jcobnew}, the above equation, the definition of a Hamiltonian vector field, the chain rule, and condition (b) in Definition~\ref{def:monotonehomotopy}, we have
\[
\begin{split}
0 & \le
\omega(\partial_sv, J\partial_sv)\\
& = \omega(\partial_sv,\partial_tv) - \omega(\partial_sv,X_{H_{s,t}})\\
&= \omega(\partial_sv,\partial_tv) + dH_{s,t}(\partial_sv)\\
&= \omega(\partial_sv,\partial_tv) + \partial_s(H\circ u) - \frac{\partial H}{\partial s}\circ u\\
& \le \omega(\partial_sv,\partial_tv) + \partial_s(H\circ u).
\end{split}
\]

It follows from the above inequality that on all of $C$ we have
\[
u^*\omega_\phi = \left(-\partial_s(H\circ u) + k\right)ds\wedge dt
\]
for some nonnegative function $k:C\to\R^{\ge 0}$. Integrating over $C$, we obtain
\[
\begin{split}
\int_Cu^*\omega_\phi &\ge \int_C\left(-\partial_s(H\circ u)\right)ds\wedge dt\\
&= \int_{f_{H_-}^{-1}(\alpha^-)} H_- dt - \int_{f_{H_+}^{-1}(\alpha^+)} H_+ dt.
\end{split}
\]
By equation \eqref{eqn:adcob}, this proves the inequality \eqref{eqn:stokesnew}.
\end{proof}

For compactness arguments, we will need:

\begin{lemma}
\label{lem:avoidboundarynew}
If $\mc{C}\in\M^J(\alpha^+,\alpha^-)$, then $\mc{C}$ does not intersect the neighborhood of $\partial (\R\times Y_\phi)$ in which \eqref{eqn:phiboundary} and \eqref{eqn:Jboundary} hold.
\end{lemma}

\begin{proof}
Let $T\subset Y_\phi$ be a torus of the form $x=\op{constant}$ in the coordinates \eqref{eqn:phiboundary}. We can choose $x$ generically so that $\mc{C}$ intersects $(\R\times T)$ transversely as a sum of closed one-dimensional submanifolds with integer weights. Since this intersection is nullhomologous in $\R\times T$, it follows from \cite[Lem.\ 3.11]{pfh3} that it is empty.
\end{proof}

We also note:

\begin{lemma}
\label{lem:carlemannew}
Let $J\in\mc{J}(\phi,H_+,H_-)$.
\begin{description}
\item{(a)}
If $C\in\mathcal{M}^J(\alpha^+,\alpha^-)$ is a somewhere injective $J$-holomorphic cylinder, then $C$ is transverse to the fibers of the projection $\pi:\R\times Y_\phi\to\R\times S^1$.
\item{(b)}
If $C$ is a somewhere injective closed $J$-holomorphic curve, then $C$ is a fiber.
\end{description}
\end{lemma}

\begin{proof}
By Remark~\ref{rem:holoproj}, the restriction of $\pi$ to $C$ defines a branched cover $C\to\R\times S^1$. Since $C$ is a cylinder, this branched cover has no ramification points, and (a) follows. Part (b) follows because a bounded holomorphic function on a compact Riemann surface is constant.
\end{proof}

Finally, similarly to Proposition~\ref{prop:ie}, as a special case of \cite[Thm.\ 4.15]{ir}, we have the following index inequality:

\begin{proposition}
\label{prop:ie2new}
Let $J \in \mc{J}(\phi,H_+,H_-)$, let $\alpha^\pm$ be nondegenerate orbit sets for $\phi_{H_\pm}$, and suppose that $C\in\mc{M}^J(\alpha^+,\alpha^-)$ does not have any multiply covered components. Then
\[
\op{ind}(C) \le I(C),
\]
with equality only if $C$ is embedded and satisfies the ECH partition conditions as in Definition~\ref{def:pc}.
\end{proposition}

Here $\op{ind}(C)$ and $I(C)$ are defined analogously to \eqref{eqn:Fredholmindex} and \eqref{eqn:ECHindex}; see \cite[\S4]{ir} for more details of these notions in symplectic cobordisms. Again, a standard transversality argument shows that a generic $J\in\mathcal{J}(\phi,H_+,H_-)$ is regular, where ``regular'' is defined analogously to Definition~\ref{def:regular}.


\section{Proof of the main theorem}

We are now prepared to begin the proof of Theorem~\ref{thm:main}.

\subsection{Moduli spaces of cylinders in cobordisms}
\label{sec:key}

Fix a boundary-admissible area-preserving diffeomorphism $\phi:\Sigma\to\Sigma$. In this subsection we consider Hamiltonians $H:Y_\phi\to\R$ with the standing assumptions that $(H_\pm)_t$ is locally constant near $\partial \Sigma$ for each $t\in S^1$, and the map $\phi_H$ is nondegenerate.

\begin{definition}
\label{def:M0}
Suppose $H_-,H_+:Y_\phi\to\R$ are Hamiltonians as above, and let $J\in\mc{J}(\phi,H_+,H_-)$. Let $\alpha^+$ and $\alpha^-$ be nondegenerate orbit sets for $\phi_{H_+}$ and $\phi_{H_-}$ respectively such that $[f_{H_+}^{-1}(\alpha^+)]=[f_{H_-}^{-1}(\alpha^-)]\in H_1(Y_\phi)$. Let $Z\in H_2(Y_\phi,f_{H_+}^{-1}(\alpha^+),f_{H_-}^{-2}(\alpha^-))$ with ECH index
\[
I(Z) = 0.
\]
Define $\mc{M}^J_0(\alpha^+,\alpha^-,Z)$ to be the set of $J$-holomorphic currents $C\in \mc{M}^J(\alpha^+,\alpha^-,Z)$ such that:
\begin{itemize}
\item
$C$ does not have any multiply covered components.
\item
Each component of $C$ is a cylinder.
\end{itemize}
Given $\epsilon> 0$, let
\[
\mc{M}_0^{J,\epsilon}(\alpha^+,\alpha^-) = \coprod_{I(Z)=0,\; \mc{A}(Z)\le \epsilon} \mc{M}^J_0(\alpha^+,\alpha^-,Z).
\]
\end{definition}

\begin{lemma}
\label{lem:cms}
In the situation of Definition~\ref{def:M0}, suppose that $J\in\mc{J}(\phi,H_+,H_-)$ extending $J_\pm$ as in Remark~\ref{rem:extending} is regular. Then:
\begin{description}
\item{(a)}
Suppose that at least one of the following hypotheses holds:
\begin{description}
\item
(i)
$J_-$ is regular, $\alpha^+$ is simple, and $\alpha^+$ is $\epsilon$-isolated for $J_+$; see Definition~\ref{def:isolation}.
\item
(ii)
$J_+$ is regular, $\alpha^-$ is simple, and $\alpha^-$ is $\epsilon$-isolated for $J_-$.
\end{description}
Then the set $\mc{M}_0^{J,\epsilon}(\alpha^+,\alpha^-)$ is finite.
\item{(b)}
Suppose that:
\begin{description}
\item{(iii)}
$\alpha^+$ or $\alpha^-$ is simple, and $\alpha^\pm$ is $\epsilon$-isolated for $J_\pm$.
\end{description}
Then for each $Z$ with $I(Z)=0$ and $\mc{A}(Z)\le\epsilon$, the mod 2 cardinality of $\mc{M}_0^J(\alpha^+,\alpha^-,Z)$ does not depend on the choice of regular $J$ extending $J_\pm$.
\end{description}
\end{lemma}

\begin{proof}
(a)
Without loss of generality, condition (i) holds.

By Proposition~\ref{prop:ie2new}, each $C\in \mc{M}_0^{J,\epsilon}(\alpha^+,\alpha^-)$ is a disjoint union of cylinders with Fredholm index $0$. Since $J$ is regular, it follows that such a curve $C$ is isolated in the moduli space of $J$-holomorphic curves. Thus $\mc{M}_0^{J,\epsilon}(\alpha^+,\alpha^-)$ is discrete, and we just need to show that it is compact.

Suppose we are given a sequence of distinct elements in $\mc{M}_0^{J,\epsilon}(\alpha^+,\alpha^-)$. Each such element is a union of cylinders locally satisfying the Floer equation \eqref{eqn:pfe}, and by \eqref{eqn:adcob} we have an upper bound on the integral of $\omega_\phi$ over each cylinder. Then by Gromov compactness as in Hamiltonian Floer theory, cf.\ \cite[\S11.1]{ad}, using Lemma~\ref{lem:avoidboundarynew} to rule out escape to the boundary of $\R\times Y_\phi$, one can pass to a subsequence that converges to a ``broken holomorphic current''. This is a tuple $(\mc{C}_k)_{k_-\le k \le k_+}$ of holomorphic currents where $k_-\le 0 \le k_+$, such that:
\begin{itemize}
\item
There are orbit sets $\alpha^-=\beta^{k_-},\ldots,\beta^0$ for $\phi_{H_-}$ such that if $k<0$ then $\mc{C}_k\in\M^{J_-}(\beta^{k+1},\beta^k)/\R$ is not a sum of trivial cylinders.
\item
$\mc{C}_0\in\mc{M}^J(\gamma^0,\beta^0)$.
\item
There are orbit sets $\gamma^0,\ldots,\gamma^{k_+}=\alpha^+$ for $\phi_{H_+}$ such that if $k>0$ then $\mc{C}_k\in\M^{J_+}(\gamma^k,\gamma^{k-1})/\R$ is not a sum of trivial cylinders.
\end{itemize}
In addition, it follows from Lemma~\ref{lem:carlemannew} (or from the proof of Gromov compactness) that each component of each $\mc{C}_k$ is a cylinder, except that when $\Sigma\simeq S^2$, so that bubbling is possible in Hamiltonian Floer theory, there can be components that are fibers of the projection $\R\times Y_\phi\to\R\times S^1$.

The action difference is additive by Remark~\ref{rem:ade}, so that
\begin{equation}
\label{eqn:ada}
\sum_{i=k_-}^{k_+}\mc{A}([\mc{C}_i]) \le \epsilon.
\end{equation}
By Lemmas~\ref{lem:actionfiltration} and \ref{lem:stokesnew}, we have
$\mathcal{A}([\mc{C}_i])\ge 0$ for each $i$.
If $k_+>0$, then by hypothesis (i), we have $\mc{A}([\mc{C}_{k_+}])>\epsilon$, contradicting \eqref{eqn:ada}. Thus $k_+=0$.

Since $\alpha^+$ is simple, $\mc{C}_0$ does not have any multiply covered components, except possibly for closed components, which by as noted above must be sphere fibers.

If $\mc{C}_0$ has no fibers, then we can apply the index inequality of Proposition~\ref{prop:ie2new} and the regularity of $J$ to conclude that $I(\mc{C}_0)\ge 0$. If $\mc{C}_0$ does contain at least one sphere fiber (possibly multiply covered), write $\mathcal{C}_0=\mathcal{C}_0'+\mathcal{F}$ where $\mathcal{C}_0'$ does not contain any fibers, and $\mathcal{F}$ is a sum of fibers with total multiplicity $m>0$. It then follows from the ``index ambiguity formula'' in \cite[Prop.\ 1.6(d)]{pfh2} that
\[
\begin{split}
I(\mathcal{C}_0) &= I(\mathcal{C}_0') + \langle c_1(E)+2\op{PD}([\alpha^\pm]),m[\Sigma]\rangle\\
&= I(\mathcal{C}_0') + 2m(d+1)
\end{split}
\]
where $d$ denotes the degree of $\alpha^\pm$. Since $I(\mathcal{C}_0')\ge 0$ as above, it follows that $I(\mathcal{C}_0)\ge 2$ in this case\footnote{In fact $I(\mathcal{C}_0)\ge 4$ because $d>0$ since otherwise no fibers could bubble off.}.

The ECH index is additive as in \eqref{eqn:additive}, so that
\[
\sum_{i=k_-}^0 I(\mc{C}_i) = 0.
\]
By Proposition~\ref{prop:fundamental}, $I(\mc{C}_i)>0$ for each $i<0$. We conclude that $k_-=0$ and $I(\mc{C}_0)=0$, and in particular $\mc{C}_0$ contains no fibers. Since $k_-=k_+=0$, the limiting broken holomorphic current is in $\mc{M}_0^{J,\epsilon}(\alpha^+,\alpha^-)$, and this completes the proof of compactness.

(b) Let $J,J'\in\mc{J}(\phi,H_+,H_-)$ be regular extensions of $J_\pm$. Choose a generic one-parameter family $\{J^\tau\}_{\tau\in[0,1]}$ of extensions of $J_\pm$ in $\mc{J}(\phi,H_+,H_-)$ with $J^0=J$ and $J^1=J'$. Define
\[
\mathcal{N} = \coprod_{\tau\in[0,1]}\{\tau\}\times\M^{J^\tau}_0(\alpha^+,\alpha^-,Z).
\]
By a standard transversality argument, if $\{J^\tau\}$ is generic then $\mathcal{N}$ is a one-dimensional manifold with\footnote{The minus sign is merely sentimental as we do not need to worry about orientations here.}
\[
\partial \mathcal{N} = \M^{J'}_0(\alpha^+,\alpha^-,Z) - \M^J_0(\alpha^+,\alpha^-,Z).
\]
To complete the proof of (b), we just need to show that $\mathcal{N}$ is compact.

Any sequence in $\mathcal{N}$ has a subsequence which converges to a pair consisting of a number $\tau\in[0,1]$ and a $J^\tau$-holomorphic broken current as in part (a). Since hypothesis (iii) holds, by the action argument in part (a), the broken holomorphic current just consists of a single holomorphic current $\mc{C}_0\in\M^{J^\tau}_0(\alpha^+,\alpha^-)$. Since $\alpha^+$ or $\alpha^-$ is simple, $\mc{C}_0$ does not have any multiply covered components except possibly for fibers. By the ECH index argument in part (a), $\mc{C}_0$ does not contain any fibers. Thus the pair $(\tau,\mc{C}_0)\in \mathcal{N}$.
\end{proof}

\begin{definition}
\label{def:N}
In the situation of Lemma~\ref{lem:cms}(a), if $I(Z)=0$ and $\mc{A}(Z)\le\epsilon$, denote the mod 2 cardinality of $\M^J_0(\alpha^+,\alpha^-,Z)$ by
\[
N^J(H_+,H_-,\alpha^+,\alpha^-,Z)\in\Z/2.
\]
In the situation of Lemma~\ref{lem:cms}(b), denote the mod 2 cardinality of $\M_0^J(\alpha^+,\alpha^-,Z)$ by
\[
N^{J_+,J_-}(H_+,H_-,\alpha^+,\alpha^-,Z)\in\Z/2.
\]
\end{definition}

We come now to the key lemma.

\begin{lemma}
\label{lem:key}
Let $H_1,H_2,H_3:Y_\phi\to\R$ be Hamiltonians\footnote{We continue to assume that $(H_i)_t$ is locally constant near $\partial\Sigma$ for each $t\in S^1$ and that $\phi_{H_i}$ is nondegenerate.} with $H_1<H_2<H_3$.
Let $\alpha^1$ and $\alpha^3$ be simple
orbit sets for $\phi_{H_1}$ and $\phi_{H_3}$ respectively. Let $J_i\in\mc{J}(\phi_{H_i})$ for $i=1,2,3$, and assume that $J_2$ is regular. Let $\epsilon>0$ and assume that $\alpha^i$ is $\epsilon$-isolated for $J_i$ for $i=1,3$. Let $Z\in H_2(W_{H_3,H_1},\alpha^3,\alpha^1)$ with $I(Z)=0$ and $\mc{A}(Z)\le\epsilon$. Let $J^-\in\mc{J}(\phi,H_2,H_1)$ and $J^+\in\mc{J}(\phi,H_3,H_2)$ be regular extensions of $J_i$. Then
\begin{gather}
\label{eqn:key}
N^{J_3,J_1}(H_3,H_1,\alpha^3,\alpha^1,Z)
= 
\quad\quad\quad\quad\quad\quad\quad\quad\quad\quad
 \quad\quad\quad\quad\quad\quad\quad\quad\quad\quad
\\
\nonumber
 \sum_{\substack{\mbox{\scriptsize $\alpha^2$ {\rm simple orbit set for} $\phi_{H_2}$}\\Z_-+Z_+=Z\\ I(Z_\pm)=0}}  N^{J_-}(H_2,H_1,\alpha^2,\alpha^1,Z_-) \cdot N^{J_+}(H_3,H_2,\alpha^3,\alpha^2,Z_+) \in \Z/2.
\end{gather}
\end{lemma}

\begin{remark}
Here is why the sum on the right hand side of \eqref{eqn:key} is finite. Since $\phi_{H_2}$ is nondegenerate, there are only finitely many simple (or non-simple) orbit sets $\alpha^2$ for $\phi_{H_2}$ with $[\alpha^2]=[\alpha^1]=[\alpha^3]$. For each such $\alpha^2$, given a nonzero term on the right hand side, we have $\mc{A}(Z_-)+\mc{A}(Z_+)=\mc{A}(Z)\le\epsilon$, and we have $\mc{A}(Z_\pm)\ge 0$ by Lemma~\ref{lem:stokesnew}, so $\mc{A}(Z_\pm)\le \epsilon$. Consequently the contribution to the right hand side of \eqref{eqn:key} from $\alpha^2$ is the mod 2 count of a subset of
\[
\mc{M}_0^{J_-,\epsilon}(\alpha^2,\alpha^1) \times \mc{M}_0^{J_+,\epsilon}(\alpha^3,\alpha^2).
\]
By Lemma~\ref{lem:cms}(a), each of the two factors in the above set is finite.
\end{remark}

\begin{remark}
A key part of the assertion of Lemma~\ref{lem:key} is that the orbit sets $\alpha^2$ in \eqref{eqn:key} are simple.
\end{remark}

\begin{proof}[Proof of Lemma~\ref{lem:key}.]
We will use a ``neck stretching'' argument. To set this up, recall from Remark~\ref{rem:extending} that there exists $s_0>0$ such that $J^-$ agrees with $g_{H_1}^*J_1$ when $s\le -s_0$ and with $g_{H_2}^*J_2$ when $s\ge s_0$, while $J^+$ agrees with $g_{H_2}^*J_2$ when $s\le -s_0$ and with $g_{H_3}^*J_3$ when $s\ge s_0$.

For $\tau\ge 0$, define an almost complex structure $J^\tau\in\mathcal{J}(\phi,H_3,H_1)$ as follows. Let $T_\tau:\R\times Y_\phi\to\R\times Y_\phi$ denote the translation sending $(s,y)\mapsto (s+\tau,y)$. We then define
\[
J^\tau = \left\{\begin{array}{cl} (T_{-2s_0-\tau})^*J^+, & s \ge -s_0-\tau,\\
(T_{2s_0+\tau})^*J^-, & s\le s_0+\tau.
\end{array}\right.
\]
Note that this definition is consistent because $J^\tau$ agrees with $g_{H_2}^*J_2$ when $-s_0-\tau\le s \le s_0+\tau$.

Now define
\[
\mathcal{N} = \coprod_{\tau\in[0,\infty)}\{\tau\}\times\M^{J^\tau}_0(\alpha^3,\alpha^1,Z).
\]
After a perturbation of the family $\{J^\tau\}$ to obtain transversality, which we omit from the notation, $\mathcal{N}$ is a one-dimensional manifold with boundary
\[
\partial \mathcal{N} = -\mc{M}^{J^0}_0(\alpha^3,\alpha^1,Z).
\]

As in the proof of Lemma~\ref{lem:cms}(b), the manifold $\mathcal{N}$ only has ends when $\tau\to\infty$. By Gromov compactness, as in the proof of Lemma~\ref{lem:cms}(a), each end converges to a broken holomorphic current consisting of:
\begin{description}
\item{(i)} A (possibly empty) tuple of holomorphic currents in $\mc{M}^{J_1}(\cdot,\cdot,\cdot)/\R$.
\item{(ii)} A holomorphic current in $\mc{M}^{J_-}(\cdot,\cdot,\cdot)$.
\item{(iii)} A (possibly empty) tuple of holomorphic currents in $\mc{M}^{J_2}(\cdot,\cdot,\cdot)/\R$.
\item{(iv)} A holomorphic current in $\mc{M}^{J_+}(\cdot,\cdot,\cdot)$.
\item{(v)} A (possibly empty) tuple of holomorphic currents in $\mc{M}^{J_3}(\cdot,\cdot,\cdot)/\R$.
\end{description}
Here the currents in (i), (iii), and (v) are not sums of trivial cylinders. Moreover, each component of each current is a cylinder, except that fiber components are a priori possible when $\Sigma\simeq S^2$.

As in the proof of Lemma~\ref{lem:cms}(a), for action reasons the broken holomorphic current contains no holomorphic currents of the form (i) or (v); and then by additivity of the ECH index and Proposition~\ref{prop:fundamental} there are no holomorphic currents of the form (iii), and the currents in (ii) and (iv) do not contain any fibers. As a result, we have a well-defined map
\begin{equation}
\label{eqn:ends}
f:\op{Ends}(\mathcal{N}) \longrightarrow \coprod_{\substack{\mbox{\scriptsize $\alpha^2$ {\rm orbit set for} $\phi_{H_2}$}\\Z_-+Z_+=Z\\ I(Z_\pm)=0}}  \mc{M}^{J_-}_0(\alpha^2,\alpha^1,Z_-) \times \mc{M}^{J_+}_0(\alpha^3,\alpha^2,Z_+).
\end{equation}

To complete the proof of \eqref{eqn:key}, we now show that for each element $(C_-,C_+)$ of the right hand side of \eqref{eqn:ends}, the number of inverse images $|f^{-1}(C_-,C_+)|$ is odd if the associated orbit set $\alpha^2$ is simple, and even otherwise. More specifically, if we write $\alpha^2=\{(\alpha_i,m_i)\}$, then we claim that
\begin{equation}
\label{eqn:countends}
\left|f^{-1}(C_-,C_+)\right|=\prod_i g_i
\end{equation}
where:
\begin{itemize}
\item
If $\alpha_i$ is elliptic, then $g_i=1$ if $m_i=1$, and $g_i=0$ if $m_i>1$.
\item
If $\alpha_i$ is positive hyperbolic, then $g_i=m_i!$.
\item
If $\alpha_i$ is negative hyperbolic, then $g_i=2^k k!$ where $k=\floor{m_i/2}$.
\end{itemize}

Recall that for each $i$, the curve $C_-$ has positive ends asymptotic to covers of $\alpha_i$ with total multiplicity $m_i$, while $C_+$ has negative ends asymptotic to covers of $\alpha_i$ with total multiplicity $m_i$. By Proposition~\ref{prop:ie2new}, the multiplicities of these covers must be given by the positive partition $p_{\alpha_i}^+(m_i)$ and the negative partition $p_{\alpha_i}^-(m_i)$, respectively.

If $(C_-,C_+)$ is in the image of the map \eqref{eqn:ends}, then it is a limit of unions of cylinders, so the two partitions $p_{\alpha_i}^+(m_i)$ and $p_{\alpha_i}^-(m_i)$ must be the same. It then follows from Lemma~\ref{lem:partitionsdisjoint} that $m_i=1$ if $\alpha_i$ is elliptic.

By the above, to prove \eqref{eqn:countends}, we can now assume that $m_i=1$ whenever $\alpha_i$ is elliptic. In this case \eqref{eqn:countends} follows from a standard gluing argument, cf.\ \cite[Prop.\ 11.2.9]{ad}. The reason for the factors $g_i$ is as follows. If $\alpha_i$ is positive hyperbolic, then by \eqref{eqn:php}, the curve $C_-$ has $m_i$ positive ends asymptotic to $\alpha_i$, while the curve $C_+$ has $m_i$ negative ends asymptotic to $\alpha_i$, and there are $m_i!$ ways to match these. If $\alpha_i$ is negative hyperbolic, then by \eqref{eqn:nhp}, the curve $C_-$ has $k$ positive ends asymptotic to the double cover of $\alpha_i$, while the curve $C_+$ has $k$ negative ends asymptotic to the double cover of $\alpha_i$, and there are $k!$ ways to match these. In addition, for each pair of ends that is matched, there are two ways to identify the sheets of the double cover of $\alpha_i$ when gluing. \end{proof}

\begin{remark}
The fact that there are an even number of ways to glue along a hyperbolic orbit with multiplicity greater than $1$, when the ECH partition conditions are satisfied, is one of the reasons for the condition on PFH (and ECH) generators in Remark~\ref{rem:generators}; see e.g.\ \cite[\S1.5]{obg1} or \cite[\S5.4]{bn}. Related considerations are important in the definition of symplectic field theory \cite{egh}, where ``bad'' Reeb orbits must be discarded; see e.g.\ \cite{bm} and \cite[\S11.1]{wendl}.
\end{remark}

\subsection{Conclusion}

\begin{proof}[Proof of Theorem~\ref{thm:main}.]
Let $\phi$ be a boundary-admissible area-preserving diffeomorphism of $\Sigma$, and let $\alpha$ be a nondegenerate simple orbit set for $\phi$. We need to find $\delta>0$ such that if $H:Y_\phi\to\R$ is a smooth function such that
\begin{description}
\item{(*)}
$H_t$ is locally constant near $\partial \Sigma$ for each $t$ and $\phi_H$ is nondegenerate and $\|H\|<\delta$
\end{description}
then there is a simple orbit set $\beta$ for $\phi_H$ such that $f_H(\alpha)$ is isotopic to $\beta$ through braids in $Y_{\phi_H}$.

Let $d$ denote the degree of $\alpha$; see Definition~\ref{def:orbitset}. Choose any $J\in\mc{J}(\phi)$; see Definition~\ref{def:Jphi}. By Lemma~\ref{lem:isolation}, we can choose $\epsilon>0$ such that $\alpha$ is strongly $\epsilon$-isolated for $J$. We claim that $\delta=\epsilon/d$ fulfills the requirements above.

To prove the claim, by Definition~\ref{def:isolation}(b), we can assume without loss of generality that $\phi$ is nondegenerate and $J$ is regular (see Definition~\ref{def:regular}) and $\alpha$ is $\epsilon$-isolated for $J$. Suppose $H: Y_\phi\to\R$ satisfies (*) above. We can assume without loss of generality, by adding a constant to $H$ if necessary, that $H$ takes values in $(0,\delta)$.

In the context of Definition~\ref{def:N}, we consider the constant Hamiltonians $H_-=0$ and $H_+=\delta$, so that $\phi_{H_-}=\phi_{H_+}=\phi$. Let $\alpha^+=\alpha^-=\alpha$, and let
\begin{equation}
\label{eqn:Ztriv}
Z=0\in H_2(Y_\phi,\alpha,\alpha).
\end{equation}
Note that the ECH index $I(Z)=0$; see \S\ref{sec:ECHindex}. Also, by \eqref{eqn:adcob}, the action $\mc{A}(Z)=\epsilon$.

We claim that
\begin{equation}
\label{eqn:N1}
N^{J,J}(\delta,0,\alpha,\alpha,Z)=1.
\end{equation}
Here the left hand side of \eqref{eqn:N1} is well-defined by Lemma~\ref{lem:cms}(b). To prove \eqref{eqn:N1}, observe that since $J\in\mc{J}(\phi)$ is regular, it follows that $J$ is also a regular element of $\mc{J}(\phi,\delta,0)$; see Definition~\ref{def:Jcobnew}. By Lemma~\ref{lem:actionfiltration},
\[
\mc{M}^J_0(\alpha,\alpha,Z) = \{\R\times\alpha\}.
\]
This set has cardinality $1$, which proves \eqref{eqn:N1}.

We now apply Lemma~\ref{lem:key} with $H_1=0$, $H_2=H$, and $H_3=\delta$, taking $\alpha^1=\alpha^3=\alpha$ and $Z$ as in \eqref{eqn:Ztriv}, and setting $J_1=J_3=J$. Here we can use any regular almost complex structures $J_2$ and $J_\pm$. By equations \eqref{eqn:key} and \eqref{eqn:N1}, there exists a simple orbit set $\alpha^2$ for $H$ such that $\mc{M}^{J_-}(\alpha^2,\alpha)\times \mc{M}^{J_+}(\alpha,\alpha^2)$ contains a pair $(C_-,C_+)$. Then $\alpha^2$ is the simple orbit set $\beta$ we are seeking. Either $C_-$ or $C_+$ (regarded as a movie of curves in $Y_\phi$ parametrized by $\R$, and then mapped to $Y_{\phi_H}$ by $f_H$) gives the required isotopy of braids, because $C_\pm$ is embedded in $\R\times Y_\phi$ by Proposition~\ref{prop:ie2new}, and transverse to the fibers of $\R\times Y_\phi\to \R\times S^1$ by Lemma~\ref{lem:carlemannew}(a).
\end{proof}



\begin{thebibliography}{99}

\bibitem{am} M.R.R. Alves and M. Meiwes, {\em Braid stability and the Hofer metric\/}, Ann. H. Lebesgue {\bf 7} (2024), 521--581.

\bibitem{ad} M. Audin and M. Damian, {\em Morse theory and Floer homology\/}, Universitext, Springer-Verlag, 2014.

\bibitem{bm} F. Bourgeois and K. Mohnke, {\em Coherent orientations in symplectic field theory\/}, Math. Z. {\bf 248} (2004), 123--146.

\bibitem{bowen} R. Bowen, {\em Entropy and the fundamental group\/}, Lecture Notes in Mathematics {\bf 668} (1978), 21--29.

\bibitem{cfhw} K. Cieliebak, A. Floer, H. Hofer, and K. Wysocki, {\em Applications of symplectic homology: Stability of the action spectrum\/}, Math. Z. {\bf 223} (1996), 27--45.

\bibitem{cghy} V. Colin, P. Ghiggini, K. Honda, and Y. Yuan, appendix to {\em Embedded contact homology and open book decompositions\/}, arXiv:1008.2734, to appear in Geom. Topol.

\bibitem{cghhl} D. Cristofaro-Gardiner, U. Hryniewicz, M. Hutchings, and H. Liu, {\em Contact three-manifolds with exactly two simple Reeb orbits\/}, Geom. Topol. {\bf 27} (2023), 3801--3831.

\bibitem{twoorbits} D. Cristofaro-Gardiner and M. Hutchings, {\em From one Reeb orbit to two\/}, J. Diff. Geom. {\bf 102} (2016), 25--36.

\bibitem{dw} A. Doan and T. Walpuski, {\em Castelnuovo's bound and rigidity in almost complex geometry\/}, Adv. Math. {\bf 379}, 107550 (2021).

\bibitem{dragnev} D. Dragnev, {\em Fredholm theory and transversality for noncompact pseudoholomorphic curves in symplectizations\/}, Comm. Pure Appl. Math. {\bf 57} (2004), 726--763.

\bibitem{pfh4} O. Edtmair and M. Hutchings, {\em PFH spectral invariants and $C^\infty$ closing lemmas\/}, arXiv:2110.02463.

\bibitem{egh} Y. Eliashberg, A. Givental, and H. Hofer, {\em Introduction to symplectic field theory\/}, GAFA (2000), Special Volume, Part II, 560--673.

\bibitem{fhs} A. Floer, H. Hofer, and D. Salamon. {\em Transversality in elliptic Morse theory for the symplectic action\/}, Duke Math. J. {\bf 80} (1995), 251--292.

\bibitem{fh} J.M. Franks and M. Handel, {\em Entropy and exponential growth of $\pi_1$ in dimension two\/}, Proc. AMS {\bf 102} (1988), 753--760.

\bibitem{hofer} H. Hofer, {\em On the topological properties of symplectic maps\/}, Proc. Roy. Soc. Edinburgh Sect. A {\bf 115} (1990), 25--38.

\bibitem{pfh2} M. Hutchings, {\em An index inequality for embedded pseudoholomorphic curves in symplectizations\/}, J. Eur. Math. Soc. {\bf 4} (2002), 313--361.

\bibitem{ir} M. Hutchings, {\em The embedded contact homology index revisited\/}, in New Perspectives and Challenges in Symplectic Field Theory, CRM Proc. Lecture Notes {\bf 49} (2009), 263--297.

\bibitem{bn} M. Hutchings, {\em Lecture notes on embedded contact homology\/}, Contact and symplectic topology, 389--484, Bolyai Soc. Math. Stud. {\bf 26}, Springer, 2014.

\bibitem{pfh3} M. Hutchings and M. Sullivan, {\em The periodic Floer homology of a Dehn twist\/}, Alg. Geom. Topol. {\bf 5} (2005), 301--354.

\bibitem{obg1} M. Hutchings and C.H. Taubes, {\em Gluing pseudoholomorphic curves along branched covered cylinders I\/}, J. Symplectic Geom. {\bf 5} (2007), 43--137.

\bibitem{katok} A. Katok and L. Mendoza, supplement to A. Katok and B. Hasselblatt, {\em Introduction to the modern theory of dynamical systems\/}, Cambridge University Press, 1995.

\bibitem{khanevsky} M. Khanevsky, {\em A gap in the Hofer metric between integral and autonomous Hamiltonian diffeomorphisms of surfaces\/}, arXiv:2205.03492.

\bibitem{nw} J. Nelson and M. Weiler, {\em Embedded contact homology of prequantization bundles\/}, J. Symplectic Geom. {\bf 21} (2023), 1077-1189.

\bibitem{newhouse} S.E. Newhouse, {\em Continuity properties of entropy\/}, Ann. Math. {\bf 129} (1989), 215--235.

\bibitem{ps} L. Polterovich and E. Shelukhin, {\em Autonomous Hamiltonian flows, Hofer's geometry and persistence modules\/}, Selecta Math. {\bf 22} (2016), 227--296.

\bibitem{siefring} R. Siefring, {\em Relative asymptotic behavior of pseudoholomorphic half-cylinders\/}, Comm. Pure Appl. Math. {\bf 61} (2008), 1631--1684.

\bibitem{siefringmultiple} R. Siefring, {\em Finite-energy pseudoholomorphic planes with multiple asymptotic limits\/}, Math. Annalen {\bf 368} (2017), 367--390.

\bibitem{taubescurrents} C. H. Taubes, {\em The structure of pseudoholomorphic subvarieties for a degenerate almost complex structure and symplectic form on $S^1\times B^3$\/}, Geom. Topol. {\bf 2} (1998), 221--332.

\bibitem{wendl} C. Wendl, {\em Lectures on symplectic field theory\/}, arXiv:1612.01009.

\bibitem{yuan} Y. Yao, {\em Computing embedded contact homology in Morse-Bott settings\/}, arXiv:2211.13876.

\bibitem{yomdin} Y. Yomdin, {\em Volume growth and entropy\/}, Israel J. Math. {\bf 57} (1987), 285--300.

\end{thebibliography}
\end{document}